\documentclass[12pt,a4paper]{article}

\usepackage[utf8]{inputenc}
\usepackage[T1]{fontenc}
\usepackage[ngerman,english]{babel}
\usepackage{lmodern}

\usepackage{amsmath}
\usepackage{amsthm}
\usepackage{amsfonts}
\usepackage{amssymb}
\usepackage{geometry}
\usepackage{color}
\usepackage{latexsym}
\usepackage{graphicx}
\usepackage{epstopdf}
\usepackage{tabularx}
\usepackage{enumerate}
\usepackage[usenames]{xcolor}
\usepackage[colorlinks,citecolor=red,linkcolor=blue]{hyperref}

\allowdisplaybreaks

\newtheorem{theorem}{Theorem}
\newtheorem{lem}{Lemma}
\newtheorem{corollary}{Corollary}

\theoremstyle{definition}
\newtheorem{definition}{Definition}

\newtheorem{remark}{Remark}

\newcommand{\norm}[1]{\left\Vert#1\right\Vert}
\newcommand{\abs}[1]{\left\vert#1\right\vert}

\newcommand{\bsk}{\boldsymbol{k}}

\newcommand{\bsx}{\boldsymbol{x}}
\newcommand{\bsz}{\boldsymbol{z}}
\newcommand{\bse}{\boldsymbol{e}}

\newcommand{\bsj}{\boldsymbol{j}}
\newcommand{\bsm}{\boldsymbol{m}}
\newcommand{\bsn}{\boldsymbol{n}}

\newcommand{\bsr}{\boldsymbol{r}}

\newcommand{\bsb}{\boldsymbol{b}}
\newcommand{\bsi}{\boldsymbol{i}}
\newcommand{\bstau}{\boldsymbol{\tau}}

\newcommand{\bsalpha}{\boldsymbol{\alpha}}
\newcommand{\bsy}{\boldsymbol{y}}
\newcommand{\bszero}{\boldsymbol{0}}

\newcommand{\cO}{\mathcal{O}}

\newcommand{\cH}{\mathcal{H}}
\newcommand{\cP}{\mathcal{P}}

\newcommand{\cB}{\mathcal{B}}

\newcommand{\rd}{\,\mathrm{d}}
\newcommand{\NN}{\mathbb{N}}
\newcommand{\ZZ}{\mathbb{Z}}
\newcommand{\RR}{\mathbb{R}}
\newcommand{\FF}{\mathbb{F}}
\newcommand{\DD}{\mathbb{D}}

\newcommand{\ee}{{\rm e}}
\newcommand{\uu}{\mathfrak{u}}

\newcommand{\supp}{\mathrm{supp}}
\newcommand{\err}{\mathrm{err}}

\newcommand{\sob}{\mathrm{sob}}

\definecolor{darkred}{RGB}{139,0,0}
\definecolor{darkgreen}{RGB}{0,100,0}
\definecolor{darkmagenta}{RGB}{139,0,139}

\begin{document}

\title{On the optimal order of integration in Hermite spaces with finite smoothness}

\author{Josef Dick\thanks{This research was supported under Australian Research Council’s Discovery Projects funding scheme (project number DP150101770).}, Christian Irrgeher, Gunther Leobacher\\ and Friedrich Pillichshammer\thanks{The authors are supported by the Austrian Science Fund (FWF): Projects F5508-N26 (Leobacher), F5509-N26 (Irrgeher and Pillichshammer) and F5506-N26 (Irrgeher), respectively, which are parts of the Special Research Program ``Quasi-Monte Carlo Methods: Theory and Applications''.}}

\maketitle

\begin{abstract}
We study the numerical approximation of integrals over $\RR^s$ with respect to the standard Gaussian measure for integrands which lie in certain Hermite spaces of functions. The decay rate of the associated sequence is specified by a single integer parameter which determines the smoothness classes and the inner product can be expressed via $L_2$ norms of the derivatives of the function.

We map higher order digital nets from the unit cube to a suitable subcube of $\mathbb{R}^s$ via a linear transformation and show that such rules achieve, apart from powers of $\log N$, the optimal rate of convergence of the integration error. 
\end{abstract}



\section{Introduction}

In this paper we study numerical integration of functions over the $s$-dimensional real space $\RR^s$ of the form 
\begin{align}\label{eq:intprob}
I_s(f)=\int_{\RR^s} f(\bsx) \varphi_s(\bsx) \rd \bsx,
\end{align}
where  $\varphi_s$ denotes the density of the $s$-dimensional standard Gaussian measure, 
\begin{align*}
\varphi_s(\bsx)=\frac{1}{(2 \pi)^{s/2}} \exp\left(-\frac{\bsx \cdot \bsx}{2}\right)\qquad \text{for } \bsx \in \RR^s.
\end{align*}
We assume that the integrands $f$ belong to a certain reproducing kernel Hilbert space $\cH_{s,\alpha}$ of smoothness $\alpha$ whose construction is based on Hermite polynomials and which is therefore called a Hermite space of smoothness $\alpha$. The exact definition of this space, which was introduced by Irrgeher and Leobacher \cite{IL}, will be given in Section~\ref{sec:fct_space}.

In order to approximate $I_s(f)$, without loss of generality, we use linear algorithms of the form $$A_{N,s}(f)=\sum_{i=1}^N w_i f(\bsx_i),$$ which are based on nodes $\bsx_1,\ldots,\bsx_N \in \RR^s$ and real weights $w_1,\ldots,w_N$ and study the worst-case absolute error $e(A_{N,s},\cH_{s,\alpha})$ of $A_{N,s}$ over the unit ball of the Hermite space, i.e. $$ e(A_{N,s},\cH_{s,\alpha}) = \sup_{\substack{f \in \cH_{s,\alpha}\\\|f\|_{s,\alpha}\leq1}} \left| I_s(f) - A_{N,s}(f) \right|. $$ The $N$-th minimal worst-case error $e(N,\cH_{s,\alpha})$ is the infimum of $e(A_{N,s},\cH_{s,\alpha})$ over all linear algorithms $A_{N,s}$ that use $N$ function values. 

For $F,G: D \subseteq \NN \rightarrow \RR$ we say $F(N) \lesssim G(N)$ if there exists some $c>0$ such that $F(N) \le c\, G(N)$ for all $N \in D$. If the positive quantity $c$ depends on some parameter, say $s$, then we may indicate this by writing $\lesssim_s$. We may use the symbol also the other way round $\gtrsim$ with the obvious meaning.

Our main result states that $e(N,\cH_{s,\alpha})$ is, up to some $\log N$-factors, of the exact order of magnitude $N^{-\alpha}$. More precisely, we show that 
\begin{align}\label{mainres}
\frac{(\log N)^{\frac{s-1}{2}}}{N^{\alpha}} \lesssim_{s,\alpha} e(N,\cH_{s,\alpha}) \lesssim_{s,\alpha}  \frac{(\log N)^{s \frac{2 \alpha+3}{4} - \frac{1}{2} }}{N^{\alpha}}.
\end{align}
For the upper bound we present an explicit algorithm. Note that we do not study about the dependence on the dimension and the smoothness of the implicit constants in \eqref{mainres}.\\

The paper is organized as follows: In the next section we will introduce the function space setting under consideration. We recall the definition of Hermite polynomials, give the definition of Hermite spaces and discuss their smoothness properties. Section~\ref{sec_int} is devoted to the numerical integration problem. After some further introductory words we will prove the lower bound from \eqref{mainres} in Subsection~\ref{subsec:lbd} (see Theorem~\ref{thm:lowerbound}). The upper bound from \eqref{mainres} will be presented in Subsections~\ref{sec:ANOVA} (Theorem~\ref{thm:upperbound}) and \ref{sec:construct} (Corollary~\ref{cor1}). In Section \ref{sec:numerics} we numerically compute the worst-case error of the presented algorithm as well as of two other types of quadrature rules and compare their performances.

\section{Hermite spaces of functions of finite smoothness}\label{sec:fct_space}

For $k \in \NN_0$, the $k$-th Hermite polynomial is given by 
\begin{align*}
H_k(x)=\frac{(-1)^k}{\sqrt{k!}} \exp(x^2/2) \frac{\rd^k}{\rd x^k} \exp(-x^2/2),
\end{align*}
which is sometimes also called normalized probabilistic Hermite polynomial, since $$\int_{\RR} H_k(x)^2 \varphi(x) \rd x=1,$$ where $\varphi$ is the standard normal density, $\varphi(x)=\frac{1}{\sqrt{2\pi}}\exp(-x^2/2)$. For example, 
\begin{align*}
H_0(x)=1,\ H_1(x)=x,\ H_2(x)=\tfrac{1}{\sqrt{2}}(x^2-1),\ H_3(x)=\tfrac{1}{\sqrt{6}}(x^3-3x),\ \ldots 
\end{align*}
Here we follow the definition given in \cite{B98}, but we remark that there are slightly different ways to introduce Hermite polynomials (see, e.g., \cite{szeg}). Note that the $\tau$th derivative of the $k$th Hermite polynomial is given by
\begin{align}\label{eq:hermitederiv}
\frac{\rd^{\tau}}{\rd x^{\tau}}H_{k}(x)=\begin{cases}\sqrt{\frac{k!}{(k-\tau)!}}\, H_{k-\tau}(x) &\text{if } k\geq\tau,\\0&\text{otherwise,}\end{cases}
\end{align}
(see, e.g., \cite{IL}). For $s \ge 2$, $\bsk=(k_1,\ldots,k_s)\in \NN_0^s$, and $\bsx=(x_1,\ldots,x_s)\in \RR^s$ we define $s$-dimensional Hermite polynomials by 
\begin{align*} 
H_{\bsk}(\bsx)=\prod_{j=1}^s H_{k_j}(x_j).
\end{align*} 

It is well-known (see~\cite{B98}) that the sequence of Hermite polynomials $\{H_{\bsk}(\bsx)\}_{\bsk \in \NN_0^s}$ forms an orthonormal basis of the function space $L^2(\RR^s,\varphi_s)$ of Gauss square-integrable functions. 
We know that for all $\bsk\in\NN_0^s$ the bound
\begin{align}\label{eq:cramer}
|H_{\bsk}(\bsx)\sqrt{\varphi_s(\bsx)}|\leq 1\qquad\text{for all }\bsx\in\RR^s
\end{align}
holds, which is a slightly weaker version of Cramer's bound (c.f.\ Sansone \cite{sansone}). The next lemma states a stronger bound on the Hermite polynomials.

\begin{lem}\label{lem:hermitebound}
For all $\bsk\in\NN_0^s$ and for all $\bsx\in\RR^s$ we have 
\begin{align}\label{eq:hermitebound}
|H_{\bsk}(\bsx)\sqrt{\varphi_s(\bsx)}|\leq \prod_{j=1}^{s}\min\left(1,\frac{\sqrt{\pi}}{k_j^{1/12}}\right).
\end{align}
\end{lem}

The proof of this lemma will be deferred to the appendix, but we would like to remark that the upper bound is sharp with respect to $\bsk$, see \cite{K04}. From this lemma it follows that 
\begin{align}\label{defsigma}
\sigma_s(\bsk):=\|H_{\bsk}\sqrt{\varphi_s}\|_\infty\le \prod_{j=1}^{s}\min\left(1,\frac{\sqrt{\pi}}{k_j^{1/12}}\right).
\end{align}

For every square-integrable $f:\RR^s\to \RR$ the $\bsk$-th Hermite coefficient of $f$ is defined as $\widehat{f}(\bsk)=\int_{\RR^s} f(\bsx) H_{\bsk}(\bsx) \varphi_s(\bsx)\rd \bsx$. Now we define the Hermite space analogous to \cite{IL}, where the construction as well as basic properties are given in great detail.

\begin{definition}\label{def:hermite}
Let $s\in \NN$ and let $r:\NN_0^s\to (0,\infty)$ be a function satisfying 
\begin{align*}
\sum_{\bsk \in \NN_0^s}r(\bsk)\sigma_s(\bsk)^2<\infty. 
\end{align*}
Then the Hermite space corresponding to $r$ is the Hilbert space
\begin{align*}
\cH_r:=\left\{f:\RR^s\to \RR: f \text{ is continuous}, \int_{\RR^s}f(\bsx)^2\varphi_s(\bsx) \rd\bsx<\infty, \|f\|_r<\infty\right\},
\end{align*}
where $\|f\|_r^2:=\sum_{\bsk\in\NN_0^s}r(\bsk)^{-1}\widehat{f}(\bsk)^2$.
The inner product  in $\cH_r$ is thus given by
\begin{align*}
\langle f,g\rangle_r =\sum_{\bsk\in\NN_0^s}\frac{1}{r(\bsk)} \widehat{f}(\bsk) \widehat{g}(\bsk).
\end{align*}
\end{definition}

This definition of a Hermite space is slightly more general than that given in \cite{IL}. There it was required $\sum_{\bsk \in \NN_0^s}r(\bsk)<\infty$. From Lemma \ref{def:hermite} it follows that $\sum_{\bsk \in \NN_0^s}r(\bsk)<\infty$ implies $\sum_{\bsk\in\NN_0^s}r(\bsk)\sigma_s(\bsk)^2<\infty$.

To see that $\cH_r$ is indeed closed under this norm one needs to show that for $f\in \cH_r$ the Hermite series for $f$ converges to a continuous function. Indeed, applying the Cauchy-Schwarz inequality,
\begin{align*}
\sum_{\bsk\in\NN_0^s} |\widehat{f}(\bsk) H_{\bsk} (\bsx)\varphi_s(\bsx)^{1/2}|&\le \left(\sum_{\bsk\in\NN_0^s}r(\bsk)\sigma_s(\bsk)^2\right)^{\frac{1}{2}}\left( \sum_{\bsk\in\NN_0^s}r(\bsk)^{-1} \widehat{f}(\bsk)^2 \frac{H_{\bsk} (\bsx)^2\varphi_s(\bsx)}{\sigma_s(\bsk)^2} \right)^{\frac{1}{2}}\\
&\le \left(\sum_{\bsk\in\NN_0^s}r(\bsk)\sigma_s(\bsk)^2\right)^{\frac{1}{2}}\|f\|_{r}<\infty.
\end{align*}
Thus $\sum_{\bsk\in\NN_0^s}\widehat{f}(\bsk) H_{\bsk} \varphi_s^{1/2}$ is a series of continuous functions which converges uniformly, so its limit is continuous. Therefore also $\sum_{\bsk\in\NN_0^s} \widehat{f}(\bsk) H_{\bsk}= \varphi_s^{-1/2}\sum_{\bsk\in\NN_0^s} \widehat{f}(\bsk) H_{\bsk} \varphi_s^{1/2}$ is continuous.\\

We are now going to define the Hermite space of smoothness $\alpha$, which are characterized by a special choice of the $r(\bsk)$ for $\bsk\in\NN_0^s$. Let $s,\alpha\in\NN$. For all $\bsk\in\NN_0^s$ we define
\begin{align}\label{eq:rsalpha}
r_{s,\alpha}(\bsk)=\prod_{j=1}^{s} r_{\alpha}(k_j)
\end{align}
with
\begin{align*}
r_{\alpha}(k)=
\begin{cases}
1 & \text{ if } k=0\\\left(\sum_{\tau=0}^{\alpha}\beta_{\tau}(k)\right)^{-1} & \text{ if } k\ge 1
\end{cases}
\end{align*}
and for integers $\tau\geq 1$,
\begin{align*}
\beta_{\tau}(k)=\begin{cases}\frac{k!}{(k-\tau)!}&\text{if }k\geq\tau,\\0&\text{otherwise.}\end{cases}
\end{align*}
Note that we have 
\begin{align*}
r_{\alpha}(k) = \left(\sum_{\tau=0}^{\min(\alpha,k)} \frac{k!}{(k-\tau)!}\right)^{-1} \le \frac{(k- \min(\alpha,k))!}{k!} = 
\begin{cases}
\frac{1}{k!} & \text{ if } 1\le k \le \alpha,\\
\frac{(k-\alpha)!}{k!} & \text{ if } k \ge \alpha. 
\end{cases}
\end{align*}

It is easily shown that $\lim_{k \rightarrow \infty} r_{\alpha}(k) k^{\alpha} = 1 $. Hence 
\begin{align*}
r_{\alpha}(k) \asymp_\alpha \frac{1}{k^\alpha}\qquad\text{for } k\in\NN.
\end{align*}
Thus $\sum_{\bsk\in\NN_0^s} r_{s,\alpha}(\bsk)\sigma_s(\bsk)^2<\infty$ for all $\alpha\in \NN$, and we may consider the associated Hermite space.

\begin{definition}\rm
We call the Hermite space $\cH_{s,\alpha}$ corresponding to $r_{s,\alpha}$ as defined in \eqref{eq:rsalpha} a {\it Hermite space with smoothness $\alpha$}. We write $\|.\|_{s,\alpha}$ and $\langle \cdot , \cdot \rangle_{s,\alpha}$ for the norm and inner product, respectively,  of $\cH_{s,\alpha}$.
\end{definition}

The name {\it Hermite space with smoothness $\alpha$} will be justified below. In the following we recall some commonly used conventions for operations with multiindices 
\begin{itemize}
\item We denote the partial derivative by $\partial_{x_i}:=\frac{\partial}{\partial x_i}$ for any $i=1,\ldots,s$.
\item We denote the mixed partial derivatives with respect to $\bsx$ by
\begin{align*}
\partial_{\bsx}^{\bstau}:=\frac{\partial^{\abs{\bstau}}}{\partial\bsx^{\bstau}}=\frac{\partial^{\tau_1}\cdots\partial^{\tau_s}}{\partial x_1^{\tau_1}\cdots\partial x_s^{\tau_s}}
\end{align*}
for any $\bstau=(\tau_1,\ldots,\tau_s)\in\NN_0^s$, where $|\bstau|=\tau_1+\cdots+\tau_s$. 
\item For vectors $\bsn=(n_1,\ldots,n_s)$ and $\bsk=(k_1,\ldots,k_s)$ we use the following notation: $$\bsn!=\prod_{j=1}^s n_j!, \ \ \binom{\bsn}{\bsk}=\prod_{j=1}^s \binom{n_j}{k_j}, \ \ |\bsn| =\sum_{j=1}^s |n_j|, \ \ \bsn \cdot \bsk=\sum_{j=1}^s n_j k_j.$$ Furthermore, $\bsn \ge \bsk$ means that $n_j \ge k_j$ for all $j \in \{1,2,\ldots,s\}$.
\end{itemize}
For $f\in\cH_{s,\alpha}$ we have the Hermite expansion, see \cite{IL},
\begin{align*}
f(\bsx)=\sum_{\bsk\in\NN_0^s}\widehat{f}(\bsk)H_{\bsk}(\bsx)\qquad\text{for all }\bsx\in\RR^s
\end{align*}
and for any $\bstau\in\NN_0^s$ with $\bstau \le \boldsymbol{\alpha}$ we have, see Lemma~\ref{le_app} in the Appendix, that
\begin{align*}
\partial_{\bsx}^{\bstau}f = \sum_{\bsk\geq\bstau}\widehat{f}(\bsk)\sqrt{\frac{\bsk!}{(\bsk-\bstau)!}}\,H_{\bsk-\bstau}.
\end{align*}

Using an analogous expression for $g$, we obtain using Parseval's theorem that
\begin{align*}
\langle f,g\rangle_{s,\alpha}&=\sum_{\bsk\in\NN_0^s}\frac{1}{r_{s,\alpha}(\bsk)}\widehat{f}(\bsk) \widehat{g}(\bsk)\\
&=\sum_{\bsk\in\NN_0^s}\prod_{j=1}^{s}\left(\sum_{\tau=0}^{\alpha}\beta_{\tau}(k_j)\right)\widehat{f}(\bsk) \widehat{g}(\bsk)\\
&=\sum_{\bsk\in\NN_0^s}\sum_{\bstau\in\{0,\ldots,\alpha\}^s}\left(\prod_{j=1}^{s}\beta_{\tau_j}(k_j)\right)\widehat{f}(\bsk) \widehat{g}(\bsk)\\
&=\sum_{\bstau\in\{0,\ldots,\alpha\}^s}\sum_{\bsk\geq \bstau}\frac{\bsk!}{(\bsk-\bstau)!}\,\widehat{f}(\bsk) \widehat{g}(\bsk)\\
&=\sum_{\bstau\in\{0,\ldots,\alpha\}^s}\int_{\RR^s}\partial_{\bsx}^{\bstau}f(\bsx)\partial_{\bsx}^{\bstau}g(\bsx)\varphi_s(\bsx)\rd\bsx.
\end{align*}
Thus the inner product of $\cH_{s,\alpha}$ can also be written as
\begin{align*}
\langle f, g \rangle_{s,\alpha} =\sum_{\bstau\in\{0,\ldots,\alpha\}^s}\int_{\RR^s}\partial_{\bsx}^{\bstau}f(\bsx)\partial_{\bsx}^{\bstau}g(\bsx) \, \varphi_s(\bsx) \rd\bsx\,.
\end{align*}
In other words, for our special function $r_{s,\alpha}$ the corresponding Hermite space is a Sobolev-type space of functions on $\RR^s$ with smoothness $\alpha$.\

\begin{remark}
Hermite spaces have already been introduced in \cite{IL} with the stronger requirement of summability of the corresponding sequence. The authors there consider the sequence $\tilde r_{s,\alpha}(\bsk)=\bsk^{-\alpha}$,
which asymptotically is same as the choice of $r_{s,\alpha}$ in this paper. But due to the stronger (and unnecessary) requirement of summability in \cite{IL} there is the restriction of $\alpha>1$, which is relaxed to 
$\alpha\geq 1$ here. 

Besides the case of polynomially decaying coefficients, Hermite spaces {\it with exponentially decaying coefficients} were also considered. Multivariate integration for such Hermite spaces has been analyzed in
\cite{IKLP14}. It is also shown there that the elements of those function spaces are analytic. 
\\

The results in \cite{IL} and \cite{IKLP14} make heavy use of the facts that Hermite spaces are reproducing kernel Hilbert spaces with canonical kernel
\begin{align}\label{eq:Kalpha}
K_{s,\alpha}(\bsx,\bsy)=\sum_{\bsk\in\NN_0^s}r(\bsk) H_{\bsk}(\bsx) H_{\bsk}(\bsy)\qquad\text{for all } \bsx,\bsy\in\RR^s.
\end{align}
The eigenfunctions of the reproducing kernel are the Hermite polynomials and the eigenvalues are precisely the numbers $r(\bsk)$. It is a curious fact that we do not make direct use of this fact here. 
\end{remark}


\section{Integration}\label{sec_int}

We are interested in numerical approximation of the values of integrals 
\begin{align*}
I_s(f)=\int_{\RR^s}f(\bsx) \varphi_s(\bsx)\rd\bsx\qquad\text{for }\quad f\in\cH_{s,\alpha}.
\end{align*} 
Without loss of generality, see, e.g., \cite[Section~4.2]{NW08} or \cite{TWW88}, we can restrict ourselves to approximating $I_s(f)$ by means of {\it linear algorithms} of the form 
\begin{align}\label{eq:linAlg}
A_{N,s}(f)=\sum_{i=1}^N w_i f(\bsx_i) \qquad\text{for }\quad f\in\cH_{s,\alpha}
\end{align}
with integration nodes $\bsx_1,\ldots,\bsx_N \in \RR^s$ and weights $w_1,\ldots,w_N \in \RR$. An important subclass of linear algorithms are quasi-Monte Carlo algorithms which are obtained by choosing the weights $w_i=1/N$ for all $1 \le i \le N$.

For $f \in\cH_{s,\alpha}$ let 
\begin{align*}
\err(f):=I_s(f)-A_{N,s}(f).
\end{align*}
The {\it worst-case error} of the algorithm $A_{N,s}$ is then defined as the worst performance of $A_{N,s}$ over the unit ball of $\cH_{s,\alpha}$, i.e., 
\begin{align}\label{eq:wcedef}
e(A_{N,s},\cH_{s,\alpha})= \sup_{\substack{f \in \cH_{s,\alpha}\\\|f\|_{s,\alpha}\leq1}}\left|\err(f)\right|.
\end{align}
Moreover, we define the $N$-th minimal worst-case error,
\begin{align*}
e(N,\cH_{s,\alpha})=\inf_{A_{N,s}}e(A_{N,s},\cH_{s,\alpha})
\end{align*}
where the infimum is taken over all linear algorithms using $N$ function evaluations. 

Numerical integration in the Hermite space has already been studied in \cite{IL}. There it has been shown that for every $N \in \NN$ there exist points $\bsx_1,\ldots,\bsx_{N} \in \RR^s$ such that the worst-case error of the quasi-Monte Carlo (QMC) algorithm $Q_{N,s}(f)=\tfrac{1}{N}\sum_{i=1}^N f(\bsx_i)$ satisfies 
\begin{align*}
e(Q_{N,s},\cH_{s,\alpha}) \lesssim_{s,\alpha} \frac{1}{\sqrt{N}}.
\end{align*}
This result, which is \cite[Corollary~3.9]{IL}, has been shown by means of an averaging argument. The convergence rate however is very weak and does not depend on  the smoothness $\alpha$. Even very large smoothness does not give information about an improved convergence rate. The aim of this paper is to improve this error estimate.


\subsection{Lower bound on the worst-case error}\label{subsec:lbd}

First we prove a lower bound on the integration error, where we use techniques initiated by Bakhvalov \cite{B59}. 

\begin{theorem}\label{thm:lowerbound}
Let $s,\alpha\in\NN$. Then for all $N\in\NN$, $N\geq 2$, the $N$th minimal worst-case error for integration in the Hermite space $\cH(K_{s,\alpha})$ is bounded from below by
\begin{align*}
e(N,K_{s,\alpha})\gtrsim_{s,\alpha} \frac{(\log N)^{\frac{s-1}{2}}}{N^{\alpha}}.
\end{align*}
\end{theorem}
\begin{proof}
Let $\cP=\{\bsx_1, \bsx_2, \ldots, \bsx_N\}$ denote the set of quadrature points used in algorithm $A_{N,s}$. For $m \in\NN_0$ we define $\DD_m=\{1, 2,\ldots, 2^m\}$ and for vectors $\bsm=(m_1,\ldots,m_s)\in\NN_0^s$ we define $\DD_{\bsm}=\DD_{m_1}\times\DD_{m_2}\times\ldots\times\DD_{m_s}$. For $i\in\DD_m$ let
\begin{align*}
h_{i,m}(x) = \begin{cases} \left(2^m x - (i-1) \right)^{\alpha} \left(i - 2^m x \right)^{\alpha} & \mbox{for } \frac{i-1}{2^m}<x<\frac{i}{2^m}, \\ 0 & \mbox{otherwise}.\end{cases}
\end{align*}
For vectors $\bsm \in \NN_0^s$, $\bsi \in \DD_{\bsm}$ and $\bsx=(x_1,\ldots,x_s)\in\RR^s$ we define
\begin{align*}
h_{\bsi, \bsm}(\bsx) = \prod_{j=1}^s h_{i_j, m_j}(x_j).
\end{align*}
Let $\supp(h_{\bsi,\bsm})$ denote the support of the function $h_{\bsi,\bsm}$ and note that $$\supp(h_{\bsi,\bsm}) \subset \prod_{j=1}^s \left(\frac{i_j-1}{2^{m_j}},\frac{i_j}{2^{m_j}}\right)\subset[0,1]^s$$ for all $\bsi\in\DD_{\bsm},\bsm\in\NN_0^s$. Let $t\in\NN$ be such that $2^{t-1} \le 2N < 2^t$. Define
\begin{align*}
h(\bsx)= \sum_{\substack{\bsm\in\NN_0^s\\ \abs{\bsm}=t}} \sum_{\substack{\bsi\in\DD_{\bsm}\\\cP\cap\supp(h_{\bsi, \bsm})=\emptyset}} h_{\bsi,\bsm}(\bsx).
\end{align*}
By definition we have $h(\bsx_n) = 0$ for all $n \in \{1,2,\ldots,N\}$ and hence also $A_{N,s}(h)=0$. 

Moreover, 
\begin{align*}
\int_{-\infty}^\infty h_{i,m}(x) \varphi(x) \rd x &= \int_{\frac{i-1}{2^m}}^{\frac{i}{2^m}} h_{i,m}(x) \varphi(x) \rd x\\
&=\int_{\frac{i-1}{2^m}}^{\frac{i}{2^m}} \left(2^m x - (i-1) \right)^\alpha \left( i - 2^m x \right)^\alpha \varphi(x)\rd x\\
&=\frac{1}{2^{m}}\int_0^1 z^\alpha (1-z)^\alpha \varphi\left(\frac{z+i-1}{2^m}\right)\rd z\\
&\geq\frac{1}{2^{m}}\int_0^1 z^\alpha (1-z)^\alpha \varphi(1)\rd z\\
&\geq\frac{1}{2^{m}}\frac{(\alpha!)^2}{(2\alpha+1)!}\frac{1}{\sqrt{2\pi\ee}},
\end{align*}
where we used the value $\beta(\alpha+1,\alpha+1)$ of the $\beta$ function. Thus we have
\begin{align}\label{inth:est}
\int_{\RR^s}h(\bsx)\varphi_s(\bsx)\rd\bsx &=\sum_{\substack{\bsm\in\NN_0^s\\\abs{\bsm}=t}}\sum_{\substack{\bsi\in\DD_{\bsm}\\\cP\cap\supp(h_{\bsi, \bsm})=\emptyset}} \prod_{j=1}^s \int_{-\infty}^\infty h_{i_j, m_j}(x_j) \varphi(x_j) \,\rd x_j \nonumber\\ 
&\geq\sum_{\substack{\bsm\in\NN_0^s\\\abs{\bsm}=t}}\sum_{\substack{\bsi\in\DD_{\bsm}\\\cP\cap\supp(h_{\bsi, \bsm})=\emptyset}}  \frac{1}{2^{t}} \frac{(\alpha!)^{2s}}{((2\alpha+1)!)^s} \frac{1}{(2\pi\,\ee)^{s/2}} \nonumber \\
&\ge \binom{t+s-1}{s-1} \frac{2^t-N}{2^{t }} \frac{(\alpha!)^{2s}}{((2\alpha+1)!)^s}\frac{1}{(2\pi\,\ee)^{s/2}} \nonumber\\ 
&> \binom{t+s-1}{s-1} \frac{1}{4} \frac{(\alpha!)^{2s}}{((2\alpha+1)!)^s} \frac{1}{(2\pi\,\ee)^{s/2}},
\end{align}
where we also used that $\cP\cap\supp(h_{\bsi, \bsm})$ is empty for at least $2^t-N$ many indices $\bsi\in\DD_{\bsm}$.

It remains to estimate the norm of the function $h$ from above. We have that
\begin{align*}
\|h\|^2_{K_{s,\alpha}}&=\sum_{\bstau \in\{0,\ldots,\alpha\}^s}\int_{\RR^s}\left(\partial^{\bstau}h(\bsx)\right)^2\varphi_s(\bsx)\rd\bsx\\
&\leq \sum_{\bstau \in\{0,\ldots,\alpha\}^s}\int_{[0,1]^s}\left(\partial^{\bstau}h(\bsx)\right)^2\rd\bsx,
\end{align*}
where we used that $\norm{\varphi_s}_{\infty}\leq 1$ and $\supp(h)\subseteq[0,1]^s$. Next we use the equivalence of the unanchored and the anchored Sobolev norm, see e.g.\, Example 2.1 in \cite{GHHR16},
\begin{align*}
\|h\|^2_{K_{s,\alpha}}&\lesssim_{s,\alpha}\sum_{\uu\subseteq[s]}\sum_{\bstau_{-\uu}\in\{1,\ldots,\alpha-1\}^{s-\abs{\uu}}}\int_{[0,1]^{\abs{\uu}}}\left(\partial^{(\bsalpha_{\uu},\bstau_{-\uu})}h(\bsx_{\uu},\bszero_{-\uu})\right)^2\rd\bsx_{\uu}\\
&\lesssim_{s,\alpha}\sum_{\uu\subseteq[s]}\sum_{\bstau_{-\uu}\in\{1,\ldots,\alpha-1\}^{s-\abs{\uu}}}\sum_{\substack{\bsm,\bsm'\in\NN_0^s\\ \abs{\bsm}=t \\ \abs{\bsm'}=t}}\sum_{\substack{\bsi\in\DD_{\bsm},\bsi'\in\DD_{\bsm'}\\ \cP\cap\supp(h_{\bsi, \bsm})=\emptyset\\ \cP\cap\supp(h_{\bsi', \bsm'})=\emptyset}}\\
&\qquad\times  \int_{[0,1]^{\abs{\uu}}}\partial^{(\bsalpha_{\uu},\bstau_{-\uu})}h_{\bsi,\bsm}(\bsx_{\uu},\bszero_{-\uu})\partial^{(\bsalpha_{\uu},\bstau_{-\uu})}h_{\bsi',\bsm'}(\bsx_{\uu},\bszero_{-\uu})\rd\bsx_{\uu},
\end{align*}
where $(\bsx_{\uu},\bszero_{-\uu})$ denotes the $s$-dimensional vector for which the $j$-th component is $x_j$ if $j\in\uu$ and $0$ if $j\notin\uu$ as well as $\partial^{(\bsalpha_{\uu},\bstau_{-\uu})}$ is the differential operator which derives $\alpha$-times with respect to the $j$th component if $j\in\uu$ and $\tau_j$-times with respect to the $j$th component if $j\notin\uu$ with $\bstau_{-\uu}=(\tau_j)_{j\notin\uu}$. Using the product form of the functions $h_{\bsi,\bsm}$ we get 
\begin{align}\label{eq:normest1}
\|h\|^2_{K_{s,\alpha}}&\lesssim_{s,\alpha}\sum_{\uu\subseteq[s]}\sum_{\bstau_{-\uu}\in\{1,\ldots,\alpha-1\}^{s-\abs{\uu}}}\sum_{\substack{\bsm,\bsm'\in\NN_0^s\\ \abs{\bsm}=t \\ \abs{\bsm'}=t}}\sum_{\substack{\bsi\in\DD_{\bsm},\bsi'\in\DD_{\bsm'}\\ \cP\cap\supp(h_{\bsi, \bsm})=\emptyset\\ \cP\cap\supp(h_{\bsi', \bsm'})=\emptyset}}\nonumber\\
&\qquad\times \prod_{j\in\uu}\int_{0}^{1}h_{i_j,m_j}^{(\alpha)}(x)h_{i_j',m_j'}^{(\alpha)}(x)\rd x \prod_{j\notin\uu}h_{i_j,m_j}^{(\tau_j)}(0)h_{i_j',m_j'}^{(\tau_j)}(0).
\end{align}

Now let $g:\RR\rightarrow\RR$ be defined by
\begin{align*}
g(x)=\begin{cases} (x-x^2)^{\alpha} & \text{for } x\in [0,1] \\ 0 & \text{otherwise.}\end{cases}
\end{align*}
For $i\in\DD_m$ and $m\in\NN$ we have that
\begin{align*}
 h_{i,m}(x)=g(2^mx-(i-1)).
\end{align*}
and for $\tau\in\{0,\ldots,\alpha-1\}$ it holds that
\begin{align*}
 h^{(\tau)}_{i,m}(0)=2^{\tau m}g^{(\tau)}(1-i)=0.
\end{align*}

Therefore, in \eqref{eq:normest1} only the term for $\uu=[s]$ is non-zero, i.e.\ we end up with
\begin{align}\label{eq:normest2}
\|h\|^2_{K_{s,\alpha}}&\lesssim_{s,\alpha}\sum_{\substack{\bsm,\bsm'\in\NN_0^s\\ \abs{\bsm}=t \abs{\bsm'}=t}}\sum_{\substack{\bsi\in\DD_{\bsm},\bsi'\in\DD_{\bsm'}\\ \cP\cap\supp(h_{\bsi, \bsm})=\emptyset\\ \cP\cap\supp(h_{\bsi', \bsm'})=\emptyset}}\prod_{j=1}^{s}\int_{0}^{1}h_{i_j,m_j}^{(\alpha)}(x)h_{i_j',m_j'}^{(\alpha)}(x)\rd x.
\end{align}

Now let $m'\leq m$ and let $i, i'$ be such that $\supp(h_{i,m}) \subseteq \supp(h_{i',m'})$, i.e.\ $2^{m-m'}(i'-1)<i\leq2^{m-m'}i'$. Then
\begin{align*}
\int_{0}^{1}h^{(\alpha)}_{i,m}(x) h^{(\alpha)}_{i',m'}(x)\rd x&=2^{\alpha(m+m')}\int_{\frac{i-1}{2^m}}^{\frac{i}{2^m}}g^{(\alpha)}(2^m x-(i-1)) g^{(\alpha)}(2^{m'} x-(i'-1)) \rd x\\
&= 2^{\alpha(m+m')-m}\int_{0}^{1}g^{(\alpha)}(z) g^{(\alpha)}\left(\frac{z+i-1}{2^{m-m'}}-(i'-1)\right)\rd z
\end{align*}
and with Lemma \ref{lem:inthelp} by setting $b=2^{m'-m}$ and $a=2^{m'-m}(i-1)-(i'-1)$, which is given in the appendix, we get
\begin{align*}
\int_{0}^{1}g^{(\alpha)}(z) g^{(\alpha)}\left(\frac{z+i-1}{2^{m-m'}}-(i'-1)\right)\rd z =\frac{(\alpha!)^2}{2\alpha+1}\, 2^{\alpha(m'-m)}.
\end{align*}
So for $m'\leq m$ it follows that the integrals in \eqref{eq:normest2} can be bounded from above by
\begin{align*}
\int_{0}^{1}h^{(\alpha)}_{i,m}(x) h^{(\alpha)}_{i',m'}(x)\rd x\lesssim_{s,\alpha}\begin{cases} 2^{2\alpha m'-m}&\text{if } \supp(h_{i,m})\subseteq \supp(h_{i',m'})\\ 0 &\text{otherwise} \end{cases}
\end{align*}

Note that for given $i'$, there are $2^{m-m'}$ values of $i \in \mathbb{D}_m$ such that $\supp(h_{i,m})\subseteq \supp(h_{i',m'})$. Thus we have
\begin{align*}
\|h\|^2_{K_{s,\alpha}}&\lesssim_{s,\alpha} \sum_{\substack{\bsm,\bsm'\in\NN_0^s\\ \abs{\bsm}=t \\ \abs{\bsm'}=t}}\sum_{\substack{\bsi\in\DD_{\bsm},\bsi'\in\DD_{\bsm'}\\\cP\cap\supp(h_{\bsi, \bsm})=\emptyset \\ \cP\cap\supp(h_{\bsi', \bsm'})=\emptyset}}\prod_{j=1}^{s}2^{2\alpha\min(m_j,m'_j)-\max(m_j,m'_j)}\\
&\lesssim_{s,\alpha} \sum_{\substack{\bsm,\bsm'\in\NN_0^s\\ \abs{\bsm}=t \\ \abs{\bsm'}=t}}\prod_{j=1}^{s}2^{2\alpha\min(m_j,m'_j)-\max(m_j,m'_j)}2^{\min(m_j,m'_j)}2^{\max(m_j,m'_j)-\min(m_j,m'_j)}\\
&=~~ \sum_{\substack{\bsm,\bsm'\in\NN_0^s\\ \abs{\bsm}=t \\ \abs{\bsm'}=t}}\prod_{j=1}^{s}2^{\alpha\left(m_j+m_j'-\abs{m_j-m_j'}\right)}\\
&\lesssim_{s,\alpha} 2^{2\alpha t}\sum_{\substack{\bsm\in\NN_0^s\\ \abs{\bsm}=t}}\sum_{\substack{\bsm'\in\NN_0^s\\ \abs{\bsm'}=t}}\prod_{j=1}^{s}2^{-\alpha\abs{m_j-m'_j}}.
\end{align*}
For any $\bsm\in\NN_0^s$ we have
\begin{align*}
\sum_{\substack{\bsm'\in\NN_0^s\\ \abs{\bsm'}=t}}\prod_{j=1}^{s}2^{-\alpha\abs{m_j-m'_j}}\leq \sum_{\substack{\bsk\in\ZZ^s\\k_1+\cdots+k_s=0}}2^{\alpha(-\abs{k_1}-\cdots-\abs{k_s})}\leq \left(\sum_{k\in\ZZ}2^{-\alpha\abs{k}}\right)^s= 3^{s\alpha}.
\end{align*}
Thus, we end up with
\begin{align}\label{normh:est2}
\|h\|^2_{K_{s,\alpha}}&\lesssim_{s,\alpha}2^{2\alpha t}\sum_{\substack{\bsm\in\NN_0^s\\ \abs{\bsm}=t}} 1 \lesssim_{s,\alpha} 2^{2 \alpha t}\binom{t+s-1}{s-1}.
\end{align}

Combining \eqref{inth:est}, the fact that $A_{N,s}(h)=0$ and \eqref{normh:est2} we finally obtain
\begin{align*}
e(A_{N,s},K_{s,\alpha})\ge \frac{|I_s(h)-A_{N,s}(h)|}{\norm{h}_{K_{s,\alpha}}}\gtrsim_{s,\alpha}  \frac{1}{2^{\alpha t}}\sqrt{\binom{t+s-1}{s-1}}\gtrsim_{s,\alpha} \frac{(\log N)^{\frac{s-1}{2}}}{N^{\alpha}}.
\end{align*}
\end{proof}

\subsection{A relation to integration in the ANOVA space}\label{sec:ANOVA}

\begin{definition}\rm
The {\it ANOVA space of smoothness $\alpha$} defined over $[0,1)^s$ (also known as {\it unanchored Sobolev space}) is given by
\begin{multline}
\cH^{\sob}_{s,\alpha}([0,1)^s):=\bigotimes_{j=1}^{s}\left\{g:[0,1)\to\RR\,:\, g^{(r)} \text{ absolutely continuous }\right.\\
\left.\text{for }r\in\{0\ldots,\alpha-1\}, g^{(\alpha)}\in L^2[0,1)\right\}
\end{multline}
with inner product
\begin{align*}
\langle g,h\rangle_{\sob,s,\alpha}&:=\sum_{\uu\subseteq\{1,\ldots,s\}}\sum_{\bstau_{\uu}\in\{0,\ldots,\alpha-1\}^{\abs{\uu}}}\\
&\qquad\times \int_{[0,1]^{s-\abs{\uu}}}\left(\int_{[0,1]^{\abs{\uu}}}\partial_{\bsz}^{(\bstau_{\uu},\alpha_{-\uu})}g(\bsz)\rd\bsz_{\uu}\right)\left(\int_{[0,1]^{\abs{\uu}}}\partial_{\bsz}^{(\bstau_{\uu},\alpha_{-\uu})}h(\bsz)\rd \bsz_{\uu}\right)\rd\bsz_{-\uu}
\end{align*}
where $\bsz_{\uu}$ denotes the $|\uu|$-dimensional vector with components $z_j$ for $j\in\uu$ and $\bsz_{-\uu}$ denotes the $(s-|\uu|)$-dimensional vector with the components $z_j$ for $j\notin\uu$. Moreover, $(\bstau_{\uu},\alpha_{-\uu})$ denotes the $s$-dimensional vector for which the $j$-th component is $\alpha$ for $j\notin\uu$ and $\tau_j$ for $j\in\uu$, where $\bstau_{\uu}=(\tau_j)_{j\in\uu}$. The norm is $\norm{\cdot}_{\sob,s,\alpha}=\sqrt{\langle\cdot,\cdot\rangle_{\sob,s,\alpha}}$.
\end{definition}

For short we write $\cH^{\sob}_{s,\alpha}:=\cH^{\sob}_{s,\alpha}([0,1)^s)$. Note that $\cH^{\sob}_{s,\alpha}$ consists of functions with domain $[0,1)^s$ instead of $\RR^s$. Also the ANOVA space of smoothness $\alpha$ is a reproducing kernel Hilbert space with kernel function $$K_{s,\alpha}^{\sob}(\bsx,\bsy)=\prod_{j=1}^s K_{\alpha}^{\sob}(x_j,y_j)$$ for $\bsx=(x_1,x_2,\ldots,x_s) \in [0,1)^s$ and similarly for $\bsy$, and where the one-dimensional kernel is given by $$K_{\alpha}^{\sob}(x,y)=\sum_{r=0}^{\alpha} \frac{B_r(x) B_r(y)}{(r!)^2}+(-1)^{\alpha+1} \frac{B_{2 \alpha}(|x-y|)}{(2 \alpha)!}$$ for $x,y \in [0,1)$, where $B_r$ denotes the Bernoulli polynomial of degree $r$.  

The worst-case absolute integration error of an algorithm $A_{N,s}$ as in \eqref{eq:linAlg} is 
\begin{align*}
e(A_{N,s},\cH_{s,\alpha}^{\sob}) = \sup_{\substack{g \in \cH_{s,\alpha}^{\sob}\\ \|g\|_{\sob,s,\alpha} \le 1}} \left|\int_{[0,1]^s} g(\bsx) \rd \bsx -A_{N,s}(g)\right|. 
\end{align*} 
Now we relate the integration problem in the Hermite space $\cH_{s,\alpha}$ to the integration problem in $\cH_{s,\alpha}^{\sob}$.

Let $Q_{N,s}$ be a QMC-rule for integration in the ANOVA space $\cH_{s,\alpha}^{\sob}$ which is based on a point set $\{\bsz_1,\bsz_2,\ldots,\bsz_N\}$ in $[0,1)^s$, i.e., 
\begin{align*}
Q_{N,s}(g)=\frac{1}{N}\sum_{n=1}^N g(\bsz_i)\ \ \ \mbox{ for } g \in \cH_{s,\alpha}^{\sob}.
\end{align*}

For any $\bsb=(b,\ldots,b)\in (0,\infty)^s$ we denote by $\cB_b$ the mapping from $[0,1]^s$ to $[-\bsb,\bsb]$ given by 
\begin{align}\label{eq:map}
\cB_b(\bsz) = 2b\bsz-\bsb.
\end{align}
Note that the mapping $\cB_b$ is just a scaling and translation of the $s$-dimensional unit cube which is fully determined by the parameter $b$. The volume of the $s$-dimensional interval $[-\bsb,\bsb]$ is then $(2b)^s$.

For integration in the Hermite space $\cH_{s,\alpha}$ we consider integration rules of the following form: let $\{\bsz_1,\bsz_2,\ldots,\bsz_N\}\subseteq[0,1)^s$ be the point set used in $Q_{N,s}$. Then we use the integration rule
\begin{align}\label{eq:intrule}
A_{N,s}(f)=\frac{(2b)^s}{N}\sum_{i=1}^{N}f(\cB_b(\bsz_i))\varphi_s(\cB_b(\bsz_i)) \ \ \ \mbox{ for } f \in \cH_{s,\alpha},
\end{align}
with $b=2\sqrt{\alpha\log N}$ for all $i\in \{1,2,\ldots,N\}$.

\begin{theorem}\label{thm:upperbound}
Let $\alpha\in\NN$ and $A_{N,s}$ be the quadrature rule defined in \eqref{eq:intrule}. Then for the worst-case error of $A_{N,s}$ in the Hermite space $\cH_{s,\alpha}$ we have
\begin{align*}
e(A_{N,s},\cH_{s,\alpha})\lesssim_{s,\alpha} (\log N)^{s \frac{2\alpha+1}{4}} e(Q_{N,s},\cH_{s,\alpha}^{\sob}) +\frac{1}{N^{\alpha}}.
\end{align*}
\end{theorem}

For the proof of Theorem~\ref{thm:upperbound} we need some tools that will be provided in the next subsection. The proof will then be given in Subsection~\ref{sec:proof_thm2}.

In Subsection~\ref{sec:construct} we provide a construction of point sets with low worst-case error $e(Q_{N,s},\cH_{s,\alpha}^{\sob})$.

\subsubsection{Auxiliary results}

\begin{lem}\label{lem:boundedness}
Let $f\in\cH_{s,\alpha}$ with $\alpha\in\NN$. Then $|f(\bsx)\sqrt{\varphi_s(\bsx)}|\lesssim_{s,\alpha} \norm{f}_{s,\alpha}$ for all $\bsx\in\RR^s$.
\end{lem}

\begin{proof}
For any $f\in\cH_{s,\alpha}$ we know that $f(\bsx)=\sum_{\bsk\in\NN_0^s} \widehat{f}(\bsk)H_{\bsk}(\bsx)$ for all $\bsx\in\RR^s$.
Using the Cauchy-Schwarz inequality and Lemma \ref{lem:hermitebound},
\begin{align*}
|f(\bsx)\sqrt{\varphi_s(\bsx)}|&\leq \sum_{\bsk\in\NN_0^s}|\widehat{f}(\bsk)| \ |H_{\bsk}(\bsx)\sqrt{\varphi_s(\bsx)}| r_{s,\alpha}(\bsk)^{-1/2}\,r_{s,\alpha}(\bsk)^{1/2}\\
&\leq \left(\sum_{\bsk\in\NN_0^s}\frac{1}{r_{s,\alpha}(\bsk)}|\widehat{f}(\bsk)|^2\right)^{1/2}  \left(\sum_{\bsk\in\NN_0^s}|H_{\bsk}(\bsx)\sqrt{\varphi_s(\bsx)}|^2 r_{s,\alpha}(\bsk)\right)^{1/2}\\
&=\norm{f}_{s,\alpha} \pi^{s/2} \left(1+\sum_{k=1}^{\infty}\frac{1}{k^{1/6}}\, r_{\alpha}(k)\right)^{s/2}.
\end{align*}
We have 
\begin{align*}
1+\sum_{k=1}^{\infty}\frac{1}{k^{1/6}}\, r_{\alpha}(k) & \le 1 + \sum_{k=1}^{\alpha} \frac{1}{k^{1/6}} \frac{1}{k!} + \sum_{k=\alpha+1}^{\infty} \frac{1}{k^{1/6}} \frac{(k-\alpha)!}{k!}\\
& \le 1 +{\rm e}-1 + \sum_{k=\alpha+1}^{\infty} \frac{1}{k^{7/6}}\\
& \le {\rm e} +\int_{\alpha}^{\infty} \frac{\rd t}{t^{7/6}} = {\rm e} + \frac{6}{\alpha^{1/6}} 
\end{align*}
and hence the desired result follows.
\end{proof}

\begin{lem}\label{lem:prodrule}
Let $f\in\cH_{s,\alpha}$. For any $\bstau\in\{0,\ldots,\alpha\}^s$ we have 
\begin{align}\label{eq:prodrule}
\partial_{\bsx}^{\bstau}\left(f\cdot\varphi_s\right)(\bsx) = \varphi_s(\bsx)\sum_{\bsj\leq\bstau}(-1)^{\abs{\bstau-\bsj}}\binom{\bstau}{\bstau-\bsj}\sqrt{(\bstau-\bsj)!}\,H_{\bstau-\bsj}(\bsx)\,\partial_{\bsx}^{\bsj}f(\bsx).
\end{align}
\end{lem}

\begin{proof}
We show this by induction on $s$. It is obvious that \eqref{eq:prodrule} holds for $\bstau=\bszero$. Now we denote by $\bse_i$ the multiindex whose $i$-th entry is $1$ with the remaining entries set to $0$. Then for any $i=1,\ldots,s$ we have for $\bstau=\bse_i$ that
\begin{align*}
\partial_{\bsx}^{\bstau}\left(f\cdot\varphi_s\right)(\bsx)&=\partial_{x_i}\left(f(\bsx)\varphi_s(\bsx)\right)\\
&=\partial_{x_i}f(\bsx)\varphi_s(\bsx)-x_i f(\bsx)\varphi_s(\bsx)\\
&=\left(H_{\bszero}(\bsx)\,\partial_{\bsx}^{\bse_i}f(\bsx)-H_{\bse_i}(\bsx)f(\bsx)\right)\varphi_s(\bsx).
\end{align*}
Next we assume that \eqref{eq:prodrule} holds for some $\bstau$. Then for any $i=1,\ldots,s$ we get that 
\begin{align*}
&\partial_{\bsx}^{\bstau+\bse_i}\left(f(\bsx)\varphi_s(\bsx)\right)=\\
&\qquad =\partial_{x_i}\!\left(\sum_{\bsj\leq\bstau}(-1)^{\abs{\bstau-\bsj}}\binom{\bstau}{\bstau-\bsj}\sqrt{(\bstau-\bsj)!}H_{\bstau-\bsj}(\bsx)\,\partial_{\bsx}^{\bsj}f(\bsx)\varphi_s(\bsx)\right)\\
&\qquad =\sum_{\bsj\leq\bstau}(-1)^{\abs{\bstau-\bsj}}\binom{\bstau}{\bstau-\bsj}\sqrt{(\bstau-\bsj)!}\,\partial_{x_i}\!\left(H_{\bstau-\bsj}(\bsx)\partial_{\bsx}^{\bsj}f(\bsx)\varphi_s(\bsx)\right)\\
&\qquad =\sum_{\bsj\leq\bstau}(-1)^{\abs{\bstau-\bsj}}\binom{\bstau}{\bstau-\bsj}\sqrt{(\bstau-\bsj)!}\left(\partial_{\bsx}^{\bse_i}H_{\bstau-\bsj}(\bsx)-x_i H_{\bstau-\bsj}(\bsx)\right)\partial_{\bsx}^{\bsj}f(\bsx)\varphi_s(\bsx)\\
&\qquad\quad+\sum_{\bsj\leq\bstau}(-1)^{\abs{\bstau-\bsj}}\binom{\bstau}{\bstau-\bsj}\sqrt{(\bstau-\bsj)!}\, H_{\bstau-\bsj}(\bsx)\,\partial^{\bsj+\bse_i}f(\bsx)\varphi_s(\bsx)\\
&\qquad =\sum_{\bsj\leq\bstau}(-1)^{\abs{\bstau-\bsj+\bse_i}}\binom{\bstau}{\bstau-\bsj}\sqrt{(\bstau-\bsj+\bse_i)!}\,H_{\bstau-\bsj+\bse_i}(\bsx)\,\partial_{\bsx}^{\bsj}f(\bsx)\varphi_s(\bsx)\\
&\qquad\quad+\sum_{0<\bsj\leq\bstau+\bse_i}(-1)^{\abs{\bstau-\bsj+\bse_i}}\binom{\bstau}{\bstau-\bsj+\bse_i}\sqrt{(\bstau-\bsj+\bse_i)!}\, H_{\bstau-\bsj+\bse_i}(\bsx)\,\partial_{\bsx}^{\bsj}f(\bsx)\varphi_s(\bsx).
\end{align*}
If we use that for all $\bsj\leq\bstau$ and $i=1,\ldots,s$,
\begin{align*}
\binom{\bstau}{\bstau-\bsj}+\binom{\bstau}{\bstau-\bsj+\bse_i}=\binom{\bstau+\bse_i}{\bstau+\bse_i-\bsj},
\end{align*}
we end up with
\begin{align*}
&\partial_{\bsx}^{\bstau+\bse_i}(f\cdot\varphi_s)(\bsx)= \varphi_s(\bsx)\sum_{\bsj\leq \bstau+\bse_i}(-1)^{\abs{\bstau-\bsj+\bse_i}}\binom{\bstau+\bse_i}{\bstau+\bse_i-\bsj}\sqrt{(\bstau+\bse_i-\bsj)!}H_{\bstau+\bse_i-\bsj}(\bsx)\partial_{\bsx}^{\bstau}f(\bsx).
\end{align*}
\end{proof}

\subsubsection{Proof of Theorem \ref{thm:upperbound}}\label{sec:proof_thm2}

We will now prove the upper bound given in Theorem \ref{thm:upperbound}. To this end let $f\in \cH_{s,\alpha}$. Then the absolute integration error can be estimated by using the triangle inequality, i.e., 
\begin{align}
\err(f)&=\abs{\int_{\RR^s}f(\bsx)\varphi_s(\bsx)\rd \bsx-\frac{(2b)^s}{N}\sum_{i=1}^{N}f(\bsx_i)\varphi_s(\bsx_i)}\nonumber\\
&\leq \abs{\int_{\RR^s\backslash[-\bsb,\bsb]}f(\bsx)\varphi_s(\bsx)\rd\bsx}+\abs{\int_{[-\bsb,\bsb]}f(\bsx)\varphi_s(\bsx)\rd\bsx-\frac{(2b)^s}{N}\sum_{i=1}^{N}f(\bsx_i)\varphi_s(\bsx_i)}\nonumber\\
&=\err_1(f)+\err_2(f),\label{eq:estimate1}
\end{align} 
where 
\begin{align*}
\err_1(f)&:=\abs{\int_{\RR^s\backslash[-\bsb,\bsb]}f(\bsx)\varphi_s(\bsx)\rd \bsx}\\
\intertext{describes the error of approximating the integral outside of $[-\bsb,\bsb]$ by zero and}
\err_2(f)&:=\abs{\int_{[-\bsb,\bsb]}f(\bsx)\varphi_s(\bsx)\rd \bsx-\frac{(2b)^s}{N}\sum_{i=1}^{N}f(\bsx_i)\varphi_s(\bsx_i)}
\end{align*}
is the integration error which results by applying the QMC rule to the function $\bsx\mapsto f(\bsx)\varphi_s(\bsx)$ restricted to the interval $[-\bsb, \bsb]$. 

\paragraph{Estimate of $\err_1(f)$:} 
With Lemma~\ref{lem:boundedness} we get
\begin{align*}
\err_1(f)&\leq \int_{\RR^s\setminus [-\bsb,\bsb]} |f(\bsx)\varphi_s(\bsx)| \rd\bsx\\
& = \int_{\RR^s\setminus [-\bsb,\bsb]} |f(\bsx) \sqrt{\varphi_s(\bsx)}| \sqrt{\varphi_s(\bsx)} \rd\bsx\\
&\lesssim_{s,\alpha}  \norm{f}_{s,\alpha} \int_{\RR^s\setminus [-\bsb,\bsb]}\frac{\exp(-\bsx \cdot \bsx/4)}{(2 \pi)^{s/2}}\rd\bsx\\
&\lesssim_{s,\alpha}  \norm{f}_{s,\alpha} \int_{[0,\infty)^s\setminus [\bszero,\bsb]}\frac{\exp(-\bsx \cdot \bsx/4)}{\pi^{s/2}}\rd\bsx.
\end{align*}
Furthermore we have that 
\begin{align*}
\int_{[0,\infty)^s\setminus [\bszero,\bsb]}\frac{\exp(-\bsx \cdot \bsx/4)}{\pi^{s/2}}\rd\bsx&=\int_{[0,\infty)^s}\frac{\exp(-\bsx \cdot \bsx/4)}{\pi^{s/2}}\rd\bsx-\int_{[\bszero,\bsb]}\frac{\exp(-\bsx \cdot \bsx/4)}{\pi^{s/2}}\rd\bsx\\
&=\left(\frac{1}{\sqrt{\pi}} \int_{0}^{\infty}\exp(-x^2/4) \rd x\right)^s-\left(\frac{1}{\sqrt{\pi}}\int_{0}^{b} \exp(-x^2/4) \rd x\right)^s\\
&=1-\left(\frac{1}{\sqrt{\pi}} \int_{0}^{b} \exp(-x^2/4) \rd x\right)^s\\
&\leq 1-\left(1-{\rm e}^{-b^2/4}\right)^s\\
& \le s \ {\rm e}^{-b^2/4}\\
& = \frac{s}{N^{\alpha}},
\end{align*}
where we used that $b=2\sqrt{\alpha\log N}$.
This shows that
\begin{align}\label{eq:err1estimate}
\err_1(f)&\lesssim_{s,\alpha} \norm{f}_{s,\alpha}\frac{1}{N^\alpha}.
\end{align}

\paragraph{Estimate of $\err_2(f)$:} 

To estimate $\err_2(f)$ we will derive an upper bound which includes the worst-case error of integration in the ANOVA space $\cH_{s,\alpha}^{\sob}$. To this end we first transform the problem from $[-\bsb,\bsb]$ to $[0,1]^s$, i.e.
\begin{align}
\err_2(f)&=\abs{\int_{[-\bsb,\bsb]}f(\bsx)\varphi_s(\bsx)\rd\bsx-\frac{(2b)^s}{N}\sum_{i=1}^{N}f(\bsx_i)\varphi_s(\bsx_i)}\nonumber\\
&=(2b)^s\abs{\int_{[0,1]^s}f(\cB_b(\bsz))\varphi_s(\cB_b(\bsz))\rd\bsz-\frac{1}{N}\sum_{i=1}^{N}f(\cB_b(\bsz_i))\varphi_s(\cB_b(\bsz_i))}.\label{eq:err2newformulation}
\end{align}

Now we need the following lemma:

\begin{lem}\label{lem:spacerelation}
Let $f\in\cH_{s,\alpha}$ and $b>0$. Then the function $g:[0,1)^s\rightarrow\RR^s$,  given by $g=(f \cdot \varphi_s)\circ \cB_b$ with $\cB_b$ as in \eqref{eq:map}, belongs to $\cH_{s,\alpha}^{\sob}$ and furthermore,
\begin{align}\label{eq:normestimate}
\norm{g}_{\sob,s,\alpha} \lesssim_{s,\alpha} b^{s(\alpha-1/2)} \norm{f}_{s,\alpha}.
\end{align}
\end{lem}

\begin{proof}
Let $f\in \cH_{s,\alpha}$. Using the Cauchy-Schwarz inequality we get
\begin{align*}
\norm{g}_{\sob,s,\alpha}^2&=\sum_{\uu\subseteq\{1,\ldots,s\}}\sum_{\bstau_{\uu}\in\{0,\ldots,\alpha-1\}^{\abs{\uu}}}\int_{[0,1]^{s-\abs{\uu}}}\left(\int_{[0,1]^{\abs{\uu}}}\partial_{\bsz}^{(\bstau_{\uu},\alpha_{-\uu})}g(\bsz)\rd\bsz_{\uu}\right)^2\rd\bsz_{-\uu}\\
&\leq \sum_{\uu\subseteq\{1,\ldots,s\}}\sum_{\bstau_{\uu}\in\{0,\ldots,\alpha-1\}^{\abs{\uu}}} \int_{[0,1]^{s}}\left(\partial_{\bsz}^{(\bstau_{\uu},\alpha_{-\uu})}g(\bsz)\right)^2\rd\bsz\\
&=\sum_{\bstau\in\{0,\ldots,\alpha\}^{s}} \int_{[0,1]^{s}}\left(\partial_{\bsz}^{\bstau}g(\bsz)\right)^2\rd\bsz.
\end{align*}
Since $g=(f\cdot \varphi_s)\circ\cB_b$, we obtain
\begin{align*}
\norm{g}_{\sob,s,\alpha}^2&\leq \sum_{\bstau\in\{0,\ldots,\alpha\}^{s}} \int_{[0,1]^{s}}\left(\partial_{\bsz}^{\bstau}(f\cdot\varphi_s)(\cB_b(\bsz))\right)^2\rd\bsz\\
&=\frac{1}{(2b)^s}\sum_{\bstau\in\{0,\ldots,\alpha\}^{s}}(2b)^{2\abs{\bstau}} \int_{[-\bsb,\bsb]}\left(\partial_{\bsx}^{\bstau}(f\cdot\varphi_s)(\bsx)\right)^2\rd\bsx,
\end{align*}
where we used $\partial_{\bsz}^{\bstau}(f\cdot\varphi_s)(\cB_b(\bsz))=(2b)^{\abs{\bstau}}\partial_{\bsx}^{\bstau}(f\cdot\varphi_s)(\bsx)$ which holds by the chain rule. From Lemma \ref{lem:prodrule} we know that
\begin{align*}
&\int_{[-\bsb,\bsb]}\left(\partial_{\bsx}^{\bstau}(f\cdot\varphi_s)(\bsx)\right)^2\rd\bsx\\
&\qquad=\int_{[-\bsb,\bsb]}\left[\sum_{\bsj\leq\bstau}(-1)^{\abs{\bstau-\bsj}}\binom{\bstau}{\bstau-\bsj}\sqrt{(\bstau-\bsj)!}\,H_{\bstau-\bsj}(\bsx)\sqrt{\varphi_s(\bsx)}\,\partial_{\bsx}^{\bsj}f(\bsx)\right]^2\varphi_s(\bsx)\rd\bsx.
\end{align*}
With Cramer's bound \eqref{eq:cramer} it follows that
\begin{align*}
\int_{[-\bsb,\bsb]}\left(\partial_{\bsx}^{\bstau}(f\cdot\varphi_s)(\bsx)\right)^2\rd\bsx&\leq\int_{[-\bsb,\bsb]}\left[\sum_{\bsj\leq\bstau}\binom{\bstau}{\bstau-\bsj}\sqrt{(\bstau-\bsj)!}\,\abs{\partial_{\bsx}^{\bsj}f(\bsx)}\right]^2\varphi_s(\bsx)\rd\bsx\\
&= \sum_{\bsj_1,\bsj_2\leq\bstau}\binom{\bstau}{\bstau-\bsj_1}\sqrt{(\bstau-\bsj_1)!}\,\binom{\bstau}{\bstau-\bsj_2}\sqrt{(\bstau-\bsj_2)!}\\
&\qquad\times\int_{[-\bsb,\bsb]}\abs{\partial_{\bsx}^{\bsj_1}f(\bsx)\partial_{\bsx}^{\bsj_2}f(\bsx)}\varphi_s(\bsx)\rd\bsx.
\end{align*}
Furthermore, using the Cauchy-Schwarz inequality we get 
\begin{align*}
&\int_{[-\bsb,\bsb]}\abs{\partial_{\bsx}^{\bsj_1}f(\bsx)\partial_{\bsx}^{\bsj_2}f(\bsx)}\varphi_s(\bsx)\rd\bsx\\
&\qquad\leq \left(\int_{[-\bsb,\bsb]}\abs{\partial_{\bsx}^{\bsj_1}f(\bsx)}^2\varphi_s(\bsx)\rd\bsx\right)^{1/2}\left(\int_{[-\bsb,\bsb]}\abs{\partial_{\bsx}^{\bsj_2}f(\bsx)}^2\varphi_s(\bsx)\rd\bsx\right)^{1/2}\\
&\qquad\leq \norm{f}_{s,\alpha}^2,
\end{align*}
and so we have, 
\begin{align*}
\int_{[-\bsb,\bsb]}\left(\partial_{\bsx}^{\bstau}(f\cdot\varphi_s)(\bsx)\right)^2\rd\bsx&\leq \norm{f}_{s,\alpha}^2\left(\sum_{\bsj\leq\bstau}\binom{\bstau}{\bstau-\bsj}\sqrt{(\bstau-\bsj)!}\right)^2\\
&=\norm{f}_{s,\alpha}^2 \prod_{i=1}^s\left(\sum_{j=0}^{\tau_i}\binom{\tau_i}{\tau_i-j}\sqrt{(\tau_i-j)!}\right)^2\\
& \le \norm{f}_{s,\alpha}^2 \prod_{i=1}^s \tau_i! 2^{2 \tau_i},
\end{align*}
where we used that
\begin{align*}
\left(\sum_{j=0}^{\tau}\binom{\tau}{\tau-j}\sqrt{(\tau-j)!}\right)^2&\leq\tau!\left(\sum_{j=0}^{\tau}\binom{\tau}{\tau-j}\right)^2=\tau!\,2^{2\tau}.
\end{align*}
Finally we get,
\begin{align*}
\norm{g}_{\sob,s,\alpha}^2&\leq \norm{f}_{s,\alpha}^2 \sum_{\bstau\in\{0,\ldots,\alpha\}^{s}}(2b)^{2\abs{\bstau}-s} \prod_{j=1}^{s}\tau_j!\,2^{2\tau_j}\\
&\leq \norm{f}_{s,\alpha}^2 \left(\frac{\alpha!}{2b} \sum_{\tau=0}^{\alpha}(4b)^{2\tau}\right)^s\\
&\lesssim_{s,\alpha} \norm{f}_{s,\alpha}^2 b^{s(2 \alpha-1)},
\end{align*}
and the result follows by taking the square-root. 
\end{proof}

We continue with estimating $\err_2(f)$. Because of Lemma~\ref{lem:spacerelation} we know that functions of the form $(f\cdot \varphi_s)\circ\cB_b$ belong to $\cH^{\sob}_{s,\alpha}$ and so we obtain from \eqref{eq:err2newformulation},
\begin{align*}
\err_2(f)&\leq (2b)^s\norm{(f\circ\cB_b)\cdot(\varphi_s\circ\cB_b)}_{\sob,s,\alpha} e(Q_{N,s},\cH_{s,\alpha}^{\sob}).
\end{align*}
With the norm estimate \eqref{eq:normestimate} and $b=2\sqrt{\alpha\log N}$ we achieve
\begin{align}
\err_2(f) 
&\lesssim_{s,\alpha} \norm{f}_{s,\alpha} (\log N)^{s \frac{2\alpha+1}{4}}e(Q_{N,s},\cH_{s,\alpha}^{\sob}). \label{eq:err2estimate}
\end{align}

\paragraph{Estimate of $\err(f)$:} 
Now we insert \eqref{eq:err1estimate} and \eqref{eq:err2estimate} into \eqref{eq:estimate1}. This way we end up with the following estimate on the absolute integration error of $f$:
\begin{align*}
\err(f)&\lesssim_{s,\alpha}  \norm{f}_{s,\alpha}  \left(\frac{1}{N^\alpha}+(\log N)^{s \frac{2\alpha+1}{4}} e(Q_{N,s},\cH_{s,\alpha}^{\sob})\right).
\end{align*}
This implies the upper bound on the worst-case error as stated in Theorem~\ref{thm:upperbound}. \hfill $\qed$

\begin{remark}
While proving the upper bound on the worst case error, we had some constants
``hidden'' in the $\lesssim_{s,\alpha}$ notation.  In some of the estimates the implied constants grow exponentially with $s$. However, in this paper we are interested on the optimal asymptotic order of magnitude of the worst-case error for $N$ tending to infinity. In this sense the dependence of the implied constants on the dimension is not an issue. This would be a matter of tractability which we leave for future research.

\end{remark}


\subsection{Digital nets} \label{sec:construct}

For the integration in the ANOVA space we use digital nets over a suitable finite field of prime order. The construction of digital nets has been introduced by Niederreiter \cite{niesiam}. For a more recent introduction into this topic we refer to \cite{DP10}. For a prime number $q$ we identify the finite field $\FF_q$ with the set $\{0,1,\ldots,q-1\}$ equipped with the usual arithmetic operations modulo $q$.

\begin{definition}\rm
Let $\FF_q$ be the finite field of prime order $q$ and let $s,m,n \in \NN$. Let $C_1,\ldots,C_s \in \FF_q^{n \times m}$ be $n \times m$ matrices over $\FF_q$.
For each $h \in \{0,\ldots,q^m-1\}$ compute the $q$-adic expansion $h=\eta_0+\eta_1 q+\cdots+\eta_{m-1} q^{m-1}$, where $\eta_0,\eta_1,\ldots,\eta_{m-1}\in \FF_q$. 
Then compute for each $j \in \{1,2,\ldots,s\}$ the matrix-vector product  $$C_j \cdot (\eta_0,\eta_1,\ldots,\eta_{m-1})^{\top}=:(\xi_{h,j,1}, \xi_{h,j,2},\ldots,\xi_{h,j,n})^{\top}$$ over $\FF_q$. Then the point set  $\cP=\{\bsx_1,\bsx_2,\ldots,\bsx_{q^m}\}$ with $\bsx_h=(x_{h,1},x_{h,2},\ldots,x_{h,s})$ and $$x_{h,j}=\frac{\xi_{h,j,1}}{q}+\frac{\xi_{h,j,2}}{q^2}+\cdots+\frac{\xi_{h,j,n}}{q^n}$$
for $h \in \{0,\ldots,q^m-1\}$ and $j \in \{1,2,\ldots,s\}$ is called a {\it digital net over $\FF_q$ with generating matrices $C_1,\ldots,C_s$}.
\end{definition}

Note that a digital net with the above parameters consists of $q^m$ elements in $[0,1)^s$. There are several powerful constructions of generating matrices for digital nets with excellent uniform distribution properties, e.g. from Sobol', Faure, Niederreiter and Niederreiter-Xing (see, e.g., \cite[Chapter~8]{DP10} and the references therein). 

\paragraph{Higher order nets according to Dick.}

In \cite{D07,D08} Dick introduced the powerful concept of digital higher order nets over $\FF_q$ for the integration in the ANOVA space of smoothness $\alpha$. It is known that such nets can achieve a convergence rate of order $(\log N)^{\alpha s}/N^{\alpha}$ where $N=q^m$ (see \cite[Section~15.6]{DP10}). For order $(2\alpha + 1)$ nets an improvement was achieved in \cite{GSY}, which shows a convergence rate of order $(\log N)^{(s-1)/2} / N^\alpha$, which is the best possible rate of convergence. For these nets we obtain the following corollary to Theorem~\ref{thm:upperbound}:

\begin{corollary}\label{cor1}
Let $\alpha\in\NN$ and $A_{N,s}$ be the quadrature rule defined in \eqref{eq:intrule} based on a higher order net of order $(2 \alpha + 1)$ over $\FF_q$ with $N=q^m$ elements. Then for the worst-case error of $A_{N,s}$ in the Hermite space $\cH_{s,\alpha}$ we have
\begin{align*}
e(A_{N,s},\cH_{s,\alpha})\lesssim_{s,\alpha} \frac{(\log{N})^{s \frac{2 \alpha+3}{4} - \frac{1}{2} }}{N^\alpha} \ \ \ \mbox{ where $N=q^m$.}
\end{align*}
\end{corollary}

Thus we have shown that the lower bound from Theorem~\ref{thm:lowerbound} is optimal up to $\log N$-factors and can be achieved using a quadrature rule of the form \eqref{eq:intrule}.

\section{Comparison with other methods}\label{sec:numerics}

So far there are two standard deterministic integration rules over $\RR^s$ with respect to the Gaussian measure, namely the {\it Gauss-Hermite rule} and transformed QMC-rules based on the inverse cumulative distribution function (CDF) of the Gaussian measure. For both methods no theoretical bounds are known which guarantee a worst-case error of order $N^{-\alpha}$ up to logarithmic terms. 

It is now interesting to compare the three methods with respect to the order of convergence numerically. Unfortunately, there is no closed form formula known for the worst-case error. So one has to use the series representation for the worst-case error which becomes numerically infeasible already for dimension $s=2$. For $s=1$ we show how to derive the series representation of the worst-case error. In the following we drop the index for the dimension, i.e.\ we write $A_{N}:=A_{N,1}$ as well as $\cH_{\alpha}:=\cH_{1,\alpha}$. 


Similar to \cite[Proposition 2.11]{DP10}, we get for the squared worst-case error that
\begin{align*}
e^2(A_{N},\cH_{\alpha})&=\int_{\RR}\int_{\RR}K(x,y)\varphi(x)\varphi(y)\rd x\rd y \\
&\qquad-2\sum_{i=1}^{N}w_i \int_{\RR}K(x,x_i)\varphi(x)\rd x +\sum_{i,j=1}^{N}w_iw_jK(x_i,x_j). 
\end{align*}
Using the kernel representation \eqref{eq:Kalpha} we obtain for the integrals,
\begin{align*}
 \int_{\RR}\int_{\RR}K(x,y)\varphi(x)\varphi(y)\rd x\rd y=\sum_{k=0}^{\infty}r_{\alpha}(k)\left(\int_{\RR}H_k(x)\varphi(x)dx\right)^2=r_{\alpha}(0)=1
\end{align*}
and for all $i=1,\ldots,N$,
\begin{align*}
 \int_{\RR}K(x,x_i)\varphi(x)\rd x =\sum_{k=0}^{\infty}r_{\alpha}(k)\int_{\RR}H_k(x)\varphi(x)dx H_k(x_i)=r_{\alpha}(0)H_0(x_i)=1
\end{align*}
Thus we get the worst-case error formula for a given algorithm with weights $w_i$ and nodes $x_i$, $i=1,\ldots,N$,
\begin{align*}
 e(A_{N},\cH_{\alpha})=\left(\left(1-\sum_{i=1}^{N}w_i\right)^2+\sum_{k=1}^{\infty}r_{\alpha}(k)\left(\sum_{i=1}^{N}w_i H_k(x_i)\right)^2\right)^{1/2}.
\end{align*}
However, for the numerical computation we have to cut off the infinite sum at some index $m\in\NN$ and we consider a truncated version of the worst-case error, given by
\begin{align*}
 e_m(A_{N},\cH_{\alpha})=\left(\left(1-\sum_{i=1}^{N}w_i\right)^2+\sum_{k=1}^{m}r_{\alpha}(k)\left(\sum_{i=1}^{N}w_i H_k(x_i)\right)^2\right)^{1/2}.
\end{align*}
In the following numerical tests this truncation parameter is chosen numerically by $m=5\times 10^7$ to keep the resulting truncation error small and negligible compared to the worst-case error. (It would be an easy exercise to prove an upper bound on the truncation error.)  

\begin{figure}[ht!]
\begin{center}
 \includegraphics[width=.8\textwidth]{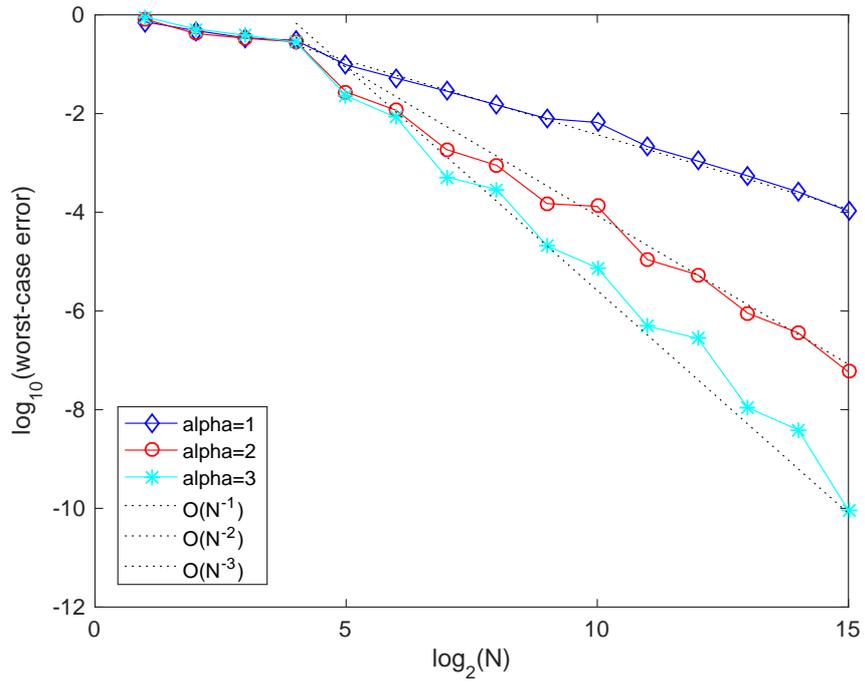}
\caption{Worst-case error of the Algorithm \eqref{eq:intrule} based on an interlaced Sobol' sequence.}
\label{fig:wceHOnets}
\end{center}
\end{figure}

First we consider the algorithm given by \eqref{eq:intrule} with a point set on the unit cube which comes from a so-called interlaced Sobol' sequence with interlacing factor equals to $\alpha$, see \cite{D07,D08}. Figure \ref{fig:wceHOnets} shows the worst-case error of this algorithm for $\alpha$ ranging from $1$ to $3$. It is an interesting observation that one can see the convergence rate $\cO(N^{-\alpha})$ of the worst-case error quite good already for rather small point sets.

\begin{figure}[ht!]
\begin{center}
 \includegraphics[width=.8\textwidth]{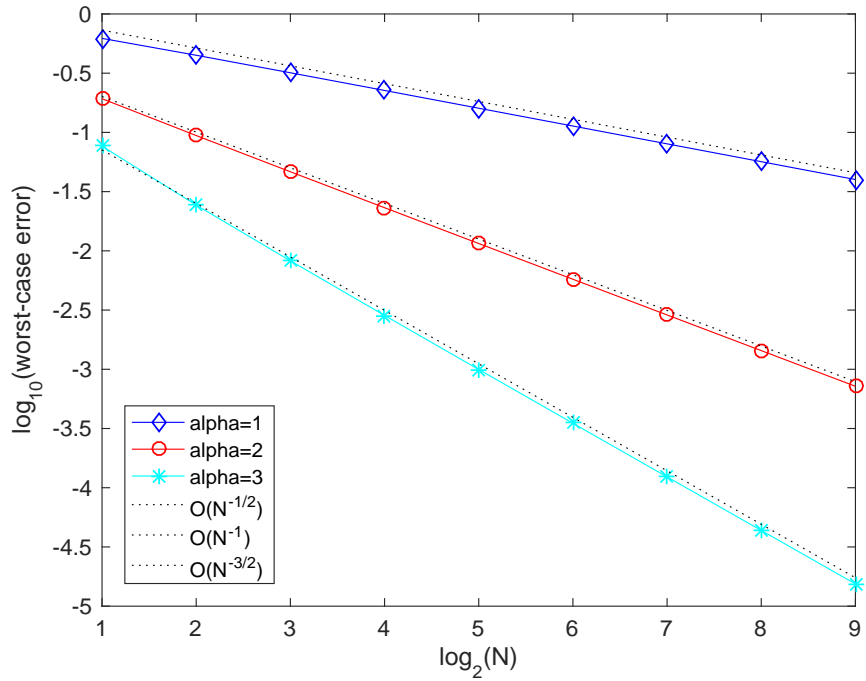}
\caption{Worst-case error of Gauss-Hermite rules.}
\label{fig:wceGH}
\end{center}
\end{figure}

A classical quadrature rule for integrals given by \eqref{eq:intprob} are the Gauss-Hermite rules, which are integration rules with integration nodes given by the roots of Hermite polynomials and with according weights. For more details on Gauss-Hermite rules we refer to \cite{hildebrand} and \cite{IKLP14}. However, Figure \ref{fig:wceGH} suggests that the worst-case error of Gauss-Hermite rules only behaves like $\cO(N^{-\alpha/2})$ which is not best possible for integration in Hermite spaces of finite smoothness, according to our lower bound.\footnote{We remark that one can show that Smolyak algorithms based on one-dimensional Gauss-Hermit rules achieve a convergence rate for the worst-case error of order at least $\cO(N^{-\alpha/2})$.} 

The other standard way of computing integrals of the form \eqref{eq:intprob} is to apply quasi-Monte Carlo integration by mapping the point set of a given QMC rule from the $s$-dimensional unit cube to the $\RR^s$ using the inverse CDF of the Gaussian measure. We remark the for this method there are no theoretical bounds known so far.

In our case here we use a higher order QMC rule based on interlaced Sobol' sequences with interlacing factor equals to $\alpha$. In Figure~\ref{fig:wceCDF} we see that for all three choices of $\alpha$ the worst-case error behaves like $\cO(N^{-1})$ which is good for $\alpha=1$, but not for $\alpha\geq 2$. It seems that the good distribution properties of the higher order nets, which guarantee the optimal convergence rate for integration in the ANOVA space defined over the unit cube, get lost by transforming the point set from $[0,1)^s$ to $\RR^s$ using the inverse CDF. This does not happen by the method presented in this paper, since the point set is just manipulated by a linear transformation which preserves the structure of the higher order nets.

\begin{figure}[ht!]
\begin{center}
 \includegraphics[width=.8\textwidth]{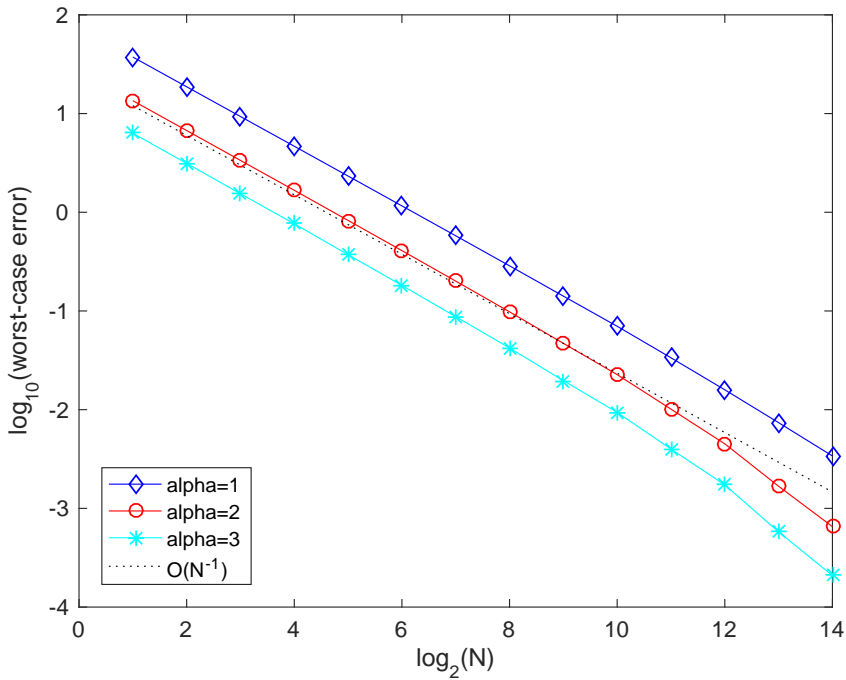}
\caption{Worst-case error of higher order quasi-Monte Carlo integration combined with the inverse cumulative distribution function.}
\label{fig:wceCDF}
\end{center}
\end{figure}

\section{Conclusions}

We introduced the notion of Hermite spaces of finite smoothness which we show to correspond to certain Sobolev-type spaces of functions on the $\RR^s$ and which are therefore of high practical interest. We considered the worst case error of integration with respect to standard Gaussian measure and we proved upper and lower bounds.

For proving the upper bound, we developed a novel and easy to implement algorithm which relies on the concept of higher order nets. The algorithm was shown to be of the optimal order in the number $N$ of integration nodes, up to factors which are polynomial in  $\log N$.

For dimension $1$ we compared numerically the convergence rate of the worst-case error of the presented method and two other standard integration rules.  

\section*{Appendix}

\begin{proof}[Proof of Lemma~\ref{lem:hermitebound}]
We first show the statement for $s=1$.  From \cite[Theorem 1]{K04} we know that for $k\geq 6$,
\begin{align*}
|H_k(x)^2\varphi(x)| \leq \frac{2}{3}C_k \exp\left(\frac{15}{8}\left(1+\frac{12}{4(2k)^{1/3}-9}\right)\right)\frac{1}{k^{1/6}}\qquad\text{for all } x\in\RR
\end{align*}
where
\begin{align*}
C_k=\begin{cases}\frac{2^{1/3}k\sqrt{4k-2}\,k!}{\sqrt{\pi}2^k\sqrt{8k^2-8k+3}\,(k/2)!^2}&\text{if } k \text{ even},\\[0.8em] \frac{\sqrt{8k^2-8k+3}\,(k-1)!}{\sqrt{\pi}2^{1/6}2^k\sqrt{2k-1}\,((k-1)/2)!^2}&\text{if } k \text{ odd}. \end{cases}
\end{align*}
It is easy to show that $C_k$ increases and converges towards $2^{1/3}/\pi$ as $k\rightarrow \infty$. Together with Cramer's bound \eqref{eq:cramer}, we get
\begin{align*}
|H_k(x)^2\varphi(x)|\leq \min\left(1,\frac{2^{4/3}}{3\pi}\exp\left(\frac{15}{8}\left(1+\frac{12}{4(2k)^{1/3}-9}\right)\right)\frac{1}{k^{1/6}}\right)\qquad\text{for all } x\in\RR.
\end{align*}
Note that the minimum is $1$ for all $k\leq 876$ and that 
\begin{align*}
\frac{2^{4/3}}{3\pi}\exp\left(\frac{15}{8}\left(1+\frac{12}{4(2k)^{1/3}-9}\right)\right)\leq 3.1\leq \pi\qquad\text{for all }k\geq 876.
\end{align*}
Thus we get for any $k\in\NN_0$ that
\begin{align*}
|H_k(x)\sqrt{\varphi(x)}| \leq \min\left(1,\frac{\sqrt{\pi}}{k^{1/12}}\right)\qquad\text{for all } x\in\RR.
\end{align*}
Now let $s\in\NN$. Then for any $\bsk\in\NN_0^s$ we get
\begin{align*}
|H_{\bsk}(\bsx)\sqrt{\varphi_s(\bsx)}|=\prod_{j=1}^{s} |H_{k_j}(x_j)\sqrt{\varphi(x_j)}|\leq \prod_{j=1}^{s}\min\left(1,\frac{\sqrt{\pi}}{k_j^{1/12}}\right)\qquad\text{for all }\bsx\in\RR^s,
\end{align*}
which shows the statement \eqref{eq:hermitebound}.
\end{proof}

\begin{lem}\label{lem:inthelp}
Let $\alpha\in\NN$ and $g:[0,1]\rightarrow\RR$ be given by $g(x)=(x-x^2)^{\alpha}$. Then for $a,b\in\RR$,
\begin{align*}
\int_{0}^{1}g^{(\alpha)}(z) g^{(\alpha)}(bz+a)\rd z =\frac{(\alpha!)^2}{2\alpha+1}\, b^{\alpha}.
\end{align*}
\end{lem}
\begin{proof}
We have that
\begin{align*}
 g^{(\alpha)}(z)=\sum_{k=0}^{\alpha}(-1)^k\binom{\alpha}{k}\frac{(\alpha+k)!}{k!}z^{k}.
\end{align*}
Thus
\begin{align*}
&\int_{0}^{1}g^{(\alpha)}(z) g^{(\alpha)}(bz+a)\rd z \\
&\qquad= \sum_{k_1=0}^{\alpha}\sum_{k_2=0}^{\alpha}(-1)^{k_1+k_2}\binom{\alpha}{k_1}\binom{\alpha}{k_2}\frac{(\alpha+k_1)!}{k_1!}\frac{(\alpha+k_2)!}{k_2!}\int_{0}^{1}z^{k_1}(bz+a)^{k_2}\rd z\\
&\qquad=\sum_{k_1=0}^{\alpha}\sum_{k_2=0}^{\alpha}(-1)^{k_1+k_2}\binom{\alpha}{k_1}\binom{\alpha}{k_2}\frac{(\alpha+k_1)!}{k_1!}\frac{(\alpha+k_2)!}{k_2!}\sum_{\ell=0}^{k_2}\binom{k_2}{\ell}b^{k_2-\ell}a^{\ell}\int_{0}^{1}z^{k_1+k_2-\ell}\rd z\\
&\qquad=\sum_{k_1=0}^{\alpha}\sum_{k_2=0}^{\alpha}(-1)^{k_1+k_2}\binom{\alpha}{k_1}\binom{\alpha}{k_2}\frac{(\alpha+k_1)!}{k_1!}\frac{(\alpha+k_2)!}{k_2!}\sum_{\ell=0}^{k_2}\binom{k_2}{\ell}b^{k_2-\ell}a^{\ell}\frac{1}{k_1+k_2-\ell+1}\\
&\qquad=(\alpha!)^2\sum_{k_2=0}^{\alpha}\sum_{\ell=0}^{k_2}(-1)^{k_2}\binom{\alpha}{k_2}\binom{\alpha+k_2}{k_2}\binom{k_2}{\ell}b^{k_2-\ell}a^{\ell}\sum_{k_1=0}^{\alpha}(-1)^{k_1}\binom{\alpha}{k_1}\binom{\alpha+k_1}{k_1}\frac{1}{k_1+k_2-\ell+1}
\end{align*}
Next we study the inner sum in more detail. For the case that $k_2-\ell=\alpha$, which only holds if $k_2=\alpha$ and $\ell=0$, we get
\begin{align*}
 \sum_{k_1=0}^{\alpha}(-1)^{k_1}\binom{\alpha}{k_1}\binom{\alpha+k_1}{k_1}\frac{1}{k_1+\alpha+1}=\frac{(-1)^{\alpha}(\alpha!)^2}{(2\alpha)!(2\alpha+1)}.
\end{align*}
For all other combinations of $k_2$ and $\ell$ we have $c:=k_2-\ell\in\{0,\ldots,\alpha-1\}$ and
\begin{align*}
 \sum_{k_1=0}^{\alpha}(-1)^{k_1}\binom{\alpha}{k_1}\binom{\alpha+k_1}{k_1}\frac{1}{k_1+c+1}=0.
\end{align*}
To sum up, we get that
\begin{align*}
 \sum_{k_1=0}^{\alpha}(-1)^{k_1}\binom{\alpha}{k_1}\binom{\alpha+k_1}{k_1}\frac{1}{k_1+k_2-\ell+1}=\begin{cases}\frac{(-1)^{\alpha}(\alpha!)^2}{(2\alpha)!(2\alpha+1)} &\text{if } k_2=\alpha \wedge \ell=0\\0&\text{otherwise}\end{cases}
\end{align*}
and, consequently,
\begin{align*}
\int_{0}^{1}g^{(\alpha)}(z) g^{(\alpha)}(bz+a)\rd z &=(\alpha!)^2(-1)^{\alpha}\binom{\alpha}{\alpha}\binom{2\alpha}{\alpha}\binom{\alpha}{0}b^{\alpha}a^{0}\frac{(-1)^{\alpha}(\alpha!)^2}{(2\alpha)!(2\alpha+1)}\\
&=\binom{2\alpha}{\alpha}\frac{(\alpha!)^2}{(2\alpha)!}\frac{(\alpha!)^2}{(2\alpha+1)}\,b^{\alpha}\\
&=\frac{(\alpha!)^2}{(2\alpha+1)}\,b^{\alpha}.
\end{align*}
\end{proof}

\begin{lem}\label{le_app}
Let $f\in \cH_{s,\alpha}$ and let $\bstau\le (\alpha,\dots,\alpha)$. Then the $\bstau$th weak derivative of $f$ exists in $L^2(\RR^s,\varphi_s)$ and it is given by
\begin{align*}
\partial_{\bsx}^{\bstau}f = \sum_{\bsk\geq\bstau}\widehat{f}(\bsk)\sqrt{\frac{\bsk!}{(\bsk-\bstau)!}}\,H_{\bsk-\bstau}\,.
\end{align*}
\end{lem}

\begin{proof}
For $\bsk\ge \bstau$ let 
$a_{\bsk-\bstau}:=\sqrt{\frac{\bsk!}{(\bsk-\bstau)!}} \widehat f(\bsk)$.
Then 
\begin{align*}
\sum_{\bsk\ge \bstau}a_{\bsk-\bstau}^2
&=\sum_{\bsk\ge \bstau} \left(\sqrt{\frac{\bsk!}{(\bsk-\bstau)!}} \widehat f(\bsk)\right)^2\\
&=\sum_{\bsk\ge \bstau} \frac{\bsk!}{(\bsk-\bstau)!}\bsr_{\bsk}\bsr_{\bsk}^{-1}\left( \widehat f(\bsk)\right)^2\\
&\le C\sum_{\bsk\ge \bstau} \frac{\bsk!}{(\bsk-\bstau)!}\bsk^{-\alpha}\bsr_{\bsk}^{-1}\left( \widehat f(\bsk)\right)^2\\
&\le C\sum_{\bsk\ge \bstau} \bsr_{\bsk}^{-1}\left( \widehat f(\bsk)\right)^2
<\infty\,.
\end{align*}
Therefore $\sum_{\bsk\ge \bstau} a_{\bsk-\bstau} H_{\bsk-\bstau}\in L^2(\RR^s,\varphi_s)$. Let $h$ be any representer of this series, that is $h$ is a measurable function on the $\RR^s$ satisfying $\int_{\RR^s} h(\bsx)^2 \varphi_s(\bsx)\rd \bsx<\infty$.


Note that from the definition of $H_k$ we have 
$$\frac{\rd^j}{\rd x^j}(H_k(x)\varphi(x))
=\frac{(-1)^k}{\sqrt{2\pi k!}}\frac{\rd^{k+j}}{\rd x^{k+j}}\exp(-x^2/2)
=(-1)^j\sqrt{\frac{(k+j)!}{k!}}H_{k+j}(x)\varphi(x).$$

For any $C^\infty$-function $\rho$ with compact support we have that $\psi=\rho/\varphi_s$ is again a $C^\infty$-function 
with compact support and
\begin{align*}
\lefteqn{
\int_{\RR^s} f(\bsx) \partial^{\bstau}_{\bsx}\rho(\bsx) \rd\bsx=
\int_{\RR^s} f(\bsx) \partial^{\bstau}_{\bsx}\left(\psi(\bsx) \varphi_s(\bsx)\right)\rd\bsx
}\\
&=\int_{\RR^s} \sum_{\bsk\ge 0}\widehat f(\bsk)H_{\bsk}(\bsx)
\sum_{\bsj\ge 0}\widehat \psi(\bsj)\partial^{\bstau}_{\bsx}\left(H_{\bsj}(\bsx)\varphi_s(\bsx)\right)\rd\bsx\\
&=\int_{\RR^s} \sum_{\bsk\ge 0}\widehat f(\bsk)H_{\bsk}(\bsx)
\sum_{\bsj\ge 0}\widehat \psi(\bsj)(-1)^{\bstau}\sqrt{\frac{(\bsj+\bstau)!}{(\bsj)!}}H_{\bsj+\bstau}(\bsx)\varphi_s(\bsx)\rd\bsx\\
&=
\sum_{\bsj\ge 0}\widehat f(\bsj+\bstau)\widehat \psi(\bsj)(-1)^{\bstau}\sqrt{\frac{(\bsj+\bstau)!}{(\bsj)!}}\\
&= (-1)^{\bstau}\sum_{\bsj\ge 0}a_{\bsj}\widehat \psi(\bsj)
= (-1)^{\bstau} \int_{\RR^s} h(\bsx) \psi(\bsx) \varphi_s(\bsx)\rd\bsx
= (-1)^{\bstau} \int_{\RR^s} h(\bsx) \rho(\bsx) \rd\bsx,
\end{align*}
such that $h$ is indeed a $\bstau$-th weak derivative of $f$.
\end{proof}

\vspace{1cm}
\noindent{\bf Authors' Address:}\\

\noindent Josef Dick\\
\noindent School of Mathematics and Statistics, The University of New South Wales, Sydney, NSW 2052, Australia.\\
\noindent Email: josef.dick@unsw.edu.au
\medskip

\noindent Christian Irrgeher\\
\noindent Johann Radon Institute for Computational and Applied Mathematics, Austrian Academy of Sciences and Department of Financial Mathematics and Applied Number Theory, Johannes Kepler University Linz, Altenbergerstra{\ss}e 69, A-4040 Linz, Austria.\\
\noindent Email: christian.irrgeher@ricam.oeaw.ac.at
\medskip

\noindent Gunther Leobacher\\
\noindent Institute of Mathematics and Scientific Computing, University of Graz, Universit\"atsplatz 3, A-8010 Graz, Austria.\\
\noindent Email: gunther.leobacher@uni-graz.at
\medskip{}

\noindent Friedrich Pillichshammer\\
\noindent Department of Financial Mathematics and Applied Number Theory, Johannes Kepler University Linz, Altenbergerstra{\ss}e 69, A-4040 Linz, Austria.\\
\noindent Email: friedrich.pillichshammer@jku.at

\end{document}